\documentclass{amsart}

\usepackage[utf8]{inputenc}
\usepackage[english]{babel}
\usepackage{amsmath}
\usepackage{amsfonts}
\usepackage{amssymb}
\usepackage{amsthm}
\usepackage{mathtools}
\usepackage[all]{xy}
\usepackage{color}
\usepackage{enumitem}

\newcounter{commentcounter}

\title[Braiding Gapped Boundaries of Dijkgraaf-Witten Theories]{Braid Group Representations from Braiding Gapped Boundaries of Dijkgraaf-Witten Theories}

\author[N. Escobar]{Nicol\'as Escobar-Velásquez}
\email{ n.escobar1726@uniandes.edu.co}
\address{Departamento de Matem\'aticas, Universidad de los Andes, Bogot\'a, Colombia.}

\author[C. Galindo]{C\'{e}sar Galindo}
\email{cn.galindo1116@uniandes.edu.co}
\address{Departamento de Matem\'aticas, Universidad de los Andes, Bogot\'a, Colombia.}

\author[Z. Wang]{Zhenghan Wang}
\email{zhenghwa@microsoft.com}
\address{Microsoft Research Station Q and Department of Mathematics,
    University of California,
    Santa Barbara, CA
    U.S.A.}

\usepackage{mathtools}
\newtheorem{theorem}{Theorem}[section]

\newtheorem{lem}[theorem]{Lemma}
\newtheorem{cor}[theorem]{Corollary}
\newtheorem{thm}[theorem]{Theorem}

\newtheorem{prop}[theorem]{Proposition}

\theoremstyle{definition}
\newtheorem{defi}[theorem]{Definition}
\newtheorem{example}[theorem]{Example}

\theoremstyle{remark}
\newtheorem{rem}[theorem]{Remark}

\newcommand{\ThreeCocycle}{\omega}
\newcommand{\bV}{\mathbf{V}}
\newcommand{\tCH}{\mathbb{C}_{\gamma}[H]}
\newcommand{\twistedCenterG}{\mathcal{Z}(G,\ThreeCocycle)}

\newcommand{\sta}{\operatorname{Sta}_G(x)}
\newcommand{\stab}{\operatorname{Sta}}
\newcommand{\Or}{\mathcal{O}}
\newcommand{\id}{\operatorname{id}}
\newcommand{\R}{\mathcal{R}}
\newcommand{\ot}{\otimes}

\newcommand{\oneOverG}{\frac{1}{|G|}}
\newcommand{\ra}{\rightarrow}
\newcommand{\Tt}{\mathbb{C}^\times}

\newcommand{\ox}{\otimes}
\newcommand{\eps}{\epsilon}

\newcommand{\Span}{\operatorname{span}}

\newcommand{\si}{\sigma_i}

\newcommand{\Av}{\operatorname{Av}}
\newcommand{\B}{\mathcal{B}}
\newcommand{\SymGp}{S}

\newcommand{\Vect}{\operatorname{Vec}}

\newcommand{\Hom}{\operatorname{Hom}}

\newcommand{\ZZ}{\mathbb{Z}}

\newcommand{\Acts}{\triangleright}
\newcommand{\Normal}{\trianglelefteq}

\begin{document}

\begin{abstract}
We study representations of the  braid groups from braiding gapped boundaries of Dijkgraaf-Witten theories and their twisted generalizations, which are (twisted) quantum doubled topological orders in two spatial dimensions.
We show that the braid representations associated to Lagrangian algebras are all monomial with respect to some specific bases. We give explicit formulas for the monomial matrices and the ground state degeneracy of the Kitaev models that are Hamiltonian realizations of Dijkgraaf-Witten theories. Our results imply that braiding gapped boundaries alone cannot provide universal gate sets for topological quantum computing with gapped boundaries.    
\end{abstract}

\thanks{Galindo was partially supported by the Ciencias B\'asicas funds from Vicerrectoria de Investigaciones de la Universidad de los Andes, Escobar was partially supported by the Convocatoria para la Financiación de Proyectos de Investigación 2017-1 funds from Facultad de Ciencias de la Universidad de los Andes, and Wang by NSF grant DMS-1411212.}

\maketitle

\section{Introduction}
Interesting new directions in topological quantum computing include its extension from anyons to gapped boundaries and symmetry defects with the hope that anyonic systems with non-universal computational power can be enhanced to achieve universality.  Enrichment of topological physics in two spatial dimensions by gapped boundaries has been investigated intensively, but their computing power has not been analyzed in detail yet. One interesting case is gapped boundaries of Dijkgraaf-Witten theories both for their experimental relevance and as theoretical exemplars (see \cite{cong2016topological,cong2017cmp,cong2017defects} and the references therein). 

In this paper, we study representations of the  braid groups from braiding gapped boundaries of Dijkgraaf-Witten theories and their twisted generalizations, which are (twisted) quantum doubled topological orders in two spatial dimensions.
We show that the resulting braid (pure braid) representations are all monomial with respect to some specific bases, hence all such representation images of the braid groups are finite groups (see also \cite{Rowell}). We give explicit formulas for the monomial matrices and the ground state degeneracy of the Kitaev models that are Hamiltonian realizations of Dijkgraaf-Witten theories. Our results imply that braiding gapped boundaries alone cannot provide universal gate sets for topological quantum computing with gapped boundaries.

For a topological order of the form $\mathcal{C}= \mathcal{Z}(\mathcal{S})$, were $\mathcal{S}$ is some unitary fusion category, gapped boundaries are modelled by Lagrangian algebras (see \cite{cong2016topological} and the references therein). For these models the ground manifolds have the form $\Hom_\mathcal{C}(1,A_1\otimes \cdots \otimes A_n)$, where the $A_i$'s are the Lagrangian algebras modelling the gapped boundaries, see \cite[Section 3]{cong2016topological} for details.  Recall that a Lagrangian algebra in any modular (tensor) category is a commutative etále algebra whose quantum dimension is maximal.  A group theoretical modular category (GTMC) is a category of the form $\mathcal{C}= \mathcal{Z}(\Vect_G^\ThreeCocycle)$ for some finite group $G$ and some $\ThreeCocycle\in Z^3(G,\Tt)$, where $\mathcal{Z}$ denotes the Drinfeld center and $\Vect_G^\ThreeCocycle$ is the category of finite dimensional $G$-graded vectors spaces with associativity constraint twisted by $\ThreeCocycle\in H^3(G,\Tt)$.

Kitaev \cite{kitaev2003fault} proposed Hamiltonian realizations of Dijkgraaf-Witten theories, who\-se topological orders are GTMCs. Moreover, extensions of these Hamiltonian realizations to surfaces with boundaries can be constructed from Lagrangian algebras \cite{bravyi1998quantum,bombin2008family,beigi2011quantum,kitaev2012models}. 

Lagrangian algebras in GTMC's are one-one correspondence with indecomposable modular categories of $\Vect_G^\ThreeCocycle$ \cite{davydov2013witt}, which are in bijection with pairs $(H,\gamma)$, where $H$ is a subgroup of $G$ and $\gamma\in C^2(H,\Tt)$ such that $\delta(\gamma)=\ThreeCocycle|_{H^{\times 3}}$, all up to conjugation \cite{natale2016equivalence}. A more direct description between Lagrangian algebras and pairs $(H,\gamma)$ can be found in \cite{davydov2010modular}. 

Recently, a quantum computing scheme to use gapped boundaries to achieve universality is proposed \cite{cong2016topological, cong2017cmp,cong2017defects,cong2017universal}. Braiding gapped boundaries can be either added to braiding anyons as in Kitaev's original proposal or as new computing primitives supplemented with other topological operations. Gapped boundaries lead to additional degeneracy to the topologically protected subspace, which potentially allows the implementation of more powerful gates. More precisely, the new gates come from representation matrices of the braid groups, $\B_n$, on objects of the GTMCs that are tensor products of Lagrangian algebras. But a characterization of the computational power of these new braid representations, mathematically a study of the representation images, was left as an important open problem \cite{cong2016topological,cong2017universal}. 

The goal of this paper is to provide such a characterization. We find a canonical monomial structure for Lagrangian algebras in $\mathcal{Z}(\Vect_G^\ThreeCocycle)$, which allows us to compute things more easily. This paper is organized as follows. Section \ref{Monomial} develops the theory of monomial representations. Specifically, it shows how to calculate invariants for a representation of $G$ using the monomial structure. In Section \ref{Section: monomial twisted Yetter-Drinfeld} we introduce the notion of monomial twisted Yetter-Drinfeld. We use  the theory developed in Section \ref{Monomial} to give an explicit description and a basis for $\Hom_{\mathcal{Z}(\Vect_G^\ThreeCocycle)}(\mathbb{C}, V^{\otimes n})$ if $V$ is a monomial object. Next, we describe the representation of $\B_n$ with respect to this basis. Theorem \ref{thm: YD monomial } states the representation is monomial and Theorem \ref{thm:action in basis} gives an explicit formula for the non-zero entries. In Section \ref{Section: representation Lagrangian} we prove that every Lagrangian algebra in $\mathcal{Z}(\Vect_G^\ThreeCocycle)$   has  a canonical monomial structure. Then the results of Section \ref{Section: monomial twisted Yetter-Drinfeld} are applied to Lagrangian algebras in $\mathcal{Z}(\Vect_G^\ThreeCocycle)$. We finish the section  developing some examples and applications.

\section{Monomial representations}\label{Monomial}

In this section we recall some basic definitions and results on monomial representations of groups.

\begin{defi}
A \textit{monomial space} is a triple $\bV=(V,X,(V_x)_{x\in X})$ where,
\begin{itemize}
\item[(i)] V is a finite dimensional complex vector space.
\item[(ii)] X is a finite set. 
\item[(iii)] $(V_x)_{x\in X}$ is a family of one dimensional subspaces of $V$, indexed by $X$, such that $ V=\bigoplus_{x\in X} V_x$.
\end{itemize}

Let $G$ be a group. By a  \textit{monomial representation} of $G$ on $\bV$ we mean a group homomorphism \[\Gamma: G\to \operatorname{GL}(V),\] such that for every $g\in G,$ $\Gamma(g)$ permutes the $V_x$'s; hence, $\Gamma$ induces an action by permutation of $G$ on $X$. We will denote $\Gamma(g)(v)$ just by $g\Acts v$.

\end{defi}

If $V$ is a representation of $G$, we denote by $V^G$
the subspace of $G$-invariant vectors, \textit{i.e.},
\[V^G=\{v\in V: g\Acts v=v, \text{ for all } g\in G\}.\]

For each $x\in X$,  we will denote $\sta$ the stabilizer of $x$ and by  $\Or_G(x)$ the $G$-orbit of $x$.  For $G$ finite, and a representation $V$ define 
\begin{align*}
    \Av_G: V \to V,&&
    v\mapsto \oneOverG \sum_{g\in G} g\Acts v.
\end{align*}
It is easy to see that $\Av_G$ is a $G$-linear projection onto $V^G$. We define,\begin{align*}
    \Av_G(V_\Or):=\Av_G(V_x),&&  x\in \Or(x),
\end{align*}
since for any $x'\in \Or_G(x)$, $\Av_G(V_x)=\Av_G(V_{x'})$.

We say that an
element $x \in X$ is \textit{regular} under the monomial action of $G$ if $\Gamma(g)$ is the identity map on $V_x$, for all $g\in \stab_G(x)$.  

Let us write $X/G$ for the set of orbits of the action of $G$ on $X$ and $\tilde{X}$ for the regular ones.

\begin{prop}\cite[Lemma 9.1]{karpilovsky1985projective}\label{Prop:Lemma Karp}
Let  $\bV=(V,X,(V_x)_{x\in X})$ be a monomial representation of $G$.
\begin{enumerate}[leftmargin=*,label=\rm{(\alph*)}]
    \item $x\in X$ is a regular element if and only if $\Av_G(V_x)\neq 0$.
    \item If $x \in X$ is a regular element under the monomial action of $G$, then
so are all elements in the $G$-orbit of $x$.
\item 
The triple 
\[
\bV^G = \left( V^G, \tilde{X}, \big (\Av_G(V_\Or)\big)_{\Or \in \tilde{X}} \right)
\]
is a monomial space. 
\item The dimension of $V^
G$ is equal to the number of regular $G$-orbits
under the monomial action of $G$ on $X$.
\end{enumerate}
\end{prop}\qed

Let $\bV=(V,X,(V_x)_{x\in X})$ and $\bV'=(V',Y,(V'_y)_{y\in Y})$ be monomial spaces. A linear isomorphism $T:V\ra V'$ is called an \textit{isomorphism of  monomial spaces} if $T(V_x)=V'_y$ for  any $x\in X$. 

\begin{prop}\label{iso of invariants}
Let $\bV=(V,X,(V_x)_{x\in X})$ and $\bV'=(V',Y,(V'_y)_{y\in Y})$ be monomial representations of a finite group $G$. If $T:V\to V'$ is a $G$-linear isomorphism of monomial spaces, then  $T|_{V^G}:\bV^G\to \bV'^G$ is an isomorphism of monomial spaces. 
\end{prop}

\begin{proof}
Clearly, $T|_{V^G}:V^G \ra V'^G$ is a linear isomorphism. 
Let $x\in X$ be a regular element. Since $T$ is an isomorphism of monomial spaces, there is some $y \in Y$ such that $T(V_x)=V'_y$. In that case: 

\[
\Av_G(V'_y)=\Av_G(T(V_x))=T(\Av_G(V_x)).
\] 
This implies $y$ is regular, because $\Av_G(V_x)\neq \{0\}$ and $T$ is an isomorphism. It also says $T|_{V^G}(\Av_G(V_{\Or(x)}))=\Av_G(V_{\Or(y)}')$, which means $T|_{V^G}$ is an isomorphism of monomial spaces. 
\end{proof}

\section{Monomial representation of the braid group}\label{Section: monomial twisted Yetter-Drinfeld}

In this section we introduce the notion of monomial twisted Yetter-Drinfeld and prove that the representation of the braid groups $\mathcal{B}_n$ over $\Hom_{\mathcal{Z}(\Vect_G^\ThreeCocycle)}(\mathbb{C},V^{\otimes n})$ is monomial if $V$ is monomial.
\subsection{Dijkgraaf-Witten theories}

Let $G$ be a discrete group. A (normalized)  3-cocycle $\ThreeCocycle \in Z^3(G, \Tt)$ is a map $\ThreeCocycle:G\times G\to \Tt$ such that 
\begin{align*}
\ThreeCocycle(ab,c,d)\ThreeCocycle(a,b,cd)&=
\ThreeCocycle(a,b,c)\ThreeCocycle(a,bc,d)\ThreeCocycle(b,c,d), & \ThreeCocycle(a,1,b)=1,
\end{align*}
for all   $a,b,c,d\in G. $ 

Let us recall the description of the modular category $\mathcal{Z}(\Vect_G^\ThreeCocycle)$, the Drinfeld center of the category $\Vect_G^\ThreeCocycle$ sometimes called the category of twisted Yetter-Drinfeld modules. The category $\mathcal{Z}(\Vect_G^\ThreeCocycle)$ is braided equivalent to the representations of the twisted Drinfeld double defined by Dijkgraaf, Pasquier and Roche in
\cite[Section 3.2]{DPR}.

Given $\ThreeCocycle \in Z^3(G; \Tt)$, we define 
\begin{align*}
  \ThreeCocycle(g,g';h)&:=\frac{\ThreeCocycle(g,^{g'}h,g')}{\ThreeCocycle(^{gg'}h,g,g')\ThreeCocycle(g,g',h)},\\
  \ThreeCocycle(g;f,h)&:=\ThreeCocycle(g,f,h)\ThreeCocycle(\prescript{g}{}{f},g,h)^{-1}\ThreeCocycle(\prescript{g}{}{f},\prescript{g}{}{h},g),
\end{align*}
for $f,g,g^\prime, h \in G$.

The objects of $\mathcal{Z}(\Vect_G^\ThreeCocycle)$ are $G$-graded vector spaces $V=\bigoplus_{g\in G} V_{g}$ with a linear map
$\triangleright: \mathbb{C}^{\ThreeCocycle}G\ot V\to V$ 
such that $1\Acts v=v$ for all $v\in V$,
\begin{align*}
(gh)\Acts v &=\ThreeCocycle(g,h;k) (g\Acts (h\Acts v)), & g,h,k & \in G, & v &\in V_k,
\end{align*}
satisfying the following compatibility condition:
\begin{align*}
g\Acts V_h & \subseteq V_{ghg^{-1}}, & g,h & \in G.
\end{align*}
Morphisms in $\mathcal{Z}(\Vect_G^\ThreeCocycle)$ are $G$-linear $G$-homogeneous maps. 
The tensor product of $V=\oplus_{g\in G}V$ and $W=\oplus_{g\in G}w$  is
$V\otimes W$ as vector space, with
\[(V\otimes W)_g=\bigoplus_{h\in G}V_h\otimes W_{h^{-1}g},\]
and for all $v\in V_g, w\in W_l$,
\[h\Acts (v\otimes w)= \ThreeCocycle(h;g,l)(h\Acts v) \otimes (h\Acts w).\]
For $V, W, Z\in\  \mathcal{Z}(\Vect_G^\ThreeCocycle)$ the associativity constrain is defined by \begin{align*}
a_{V,W,Z}: (V\ot W)\ot Z&\to V\ot (W\ot Z)\\ 
(v_g\ot w_h)\ot z_k &\mapsto \ThreeCocycle(g,h,k) v_g\ot( w_h\ot z_k)
\end{align*}
for all $g,h,k \in G, v_g\in V_x, w_h\in W_h, z_k\in Z_k$. The category is tensor braided, with braiding $c_{V,W}:V\ot W\to W\ot V$, $V,W\in\mathcal{Z}(\Vect_G^\ThreeCocycle)$,
\begin{align*}
c_{V,W}(v\ot w) &= (g\Acts w) \ot v, & g & \in G, & v &\in V_g, & w&\in W.
\end{align*}

\subsection{Braid group representation of twisted Yetter-Drinfeld modules}

Since the braided category $\mathcal{Z}(\Vect_G^\ThreeCocycle)$ is not strict,  we must be careful about the way we associate terms when we consider tensor products with more than two objects. For a list of objects $A_1, A_2, \ldots, A_n \in \mathcal{Z}(\Vect_G^\ThreeCocycle)$ we define
\[A_1\ox \cdots \ox A_n := (\cdots (A_1\ox A_2)\ox \cdots \ox A_{n-1})\ox A_n, \] 
and an isomorphism by
\begin{align}
\si' = (a^{-1}_{A_1  \ox \cdots \ox A_{i-1},A_{i+1},A_{i}}&\ox \id_{A_{i+2} \ox \cdots \ox A_{n}})\circ \label{generadores}\\
(\id_{A_1\ox \cdots A_{i-1}}&\ox c_{A_i,A_{i+1}}\ox \id_{A_{i+2} \ox \cdots \ox A_{n}})\circ \notag\\
 (&a_{A_1  \ox \cdots \ox A_{i-1},A_i,A_{i+1}}\ox \id_{A_{i+2} \ox \cdots \ox A_{n}}), \notag
\end{align}where $a_{V,W,Z}$ denotes the associativity constrains.

If $A=A_1=\cdots =A_n$, there exists a unique group homomorphism $\rho_n:\mathcal{B}_n\to \operatorname{Aut}_{\mathcal{Z}(\Vect_G^\ThreeCocycle)}(A^{\otimes n})$ sending the generator $\sigma_i \in \mathcal{B}_n$ to $\sigma'_i$. 

In general, the pure braid group $\mathcal{P}_n$ acts on $A_1\ot \cdots \ot A_n$, in the sense that there exists group homomorphism $\rho_n:\mathcal{P}_n\to \operatorname{Aut}_{\mathcal{Z}(\Vect_G^\ThreeCocycle)}(A_1\ot \cdots \ot A_n)$. 
\subsection{Crossed $G$-sets}

Let $G$ be a group. A (left) \textit{crossed} $G$-\textit{set} is a left $G$-set $X$ and a grading function $|-|:X\to G$ such that 
\[|gx|=g|x|g^{-1}\]for all $x\in X, g\in G$.
If  $X$ and $Y$ are crossed $G$-sets, a  $G$-equivariant map $f:X\to Y$ is a morphism of crossed $G$-sets if  $|f(x)|=|x|$ for all $x\in X$.

If $X$ and $Y$ are crossed $G$-sets, the cartesian product $X\times Y$ is a crossed $G$-set with the diagonal action and grading map$|(x,y)|=|x||y|$.

The category of crossed $G$-sets is a braided category with braiding 
\begin{align*}
    c_{X,Y}:X\times Y&\to Y\times X\\
     (x,y)&\mapsto (|x|\Acts y,x).
\end{align*}

Thus, given a crossed  $G$-set $X$ the braid group $\mathcal{B}_n$ acts on $X^{n}$, in the following way
\[\sigma_i':=\id_{X^{ i-1}}\times c_{X,X}\times \id_{X^{n-(i-1)}}.\]

\subsection{Monomial objects of $\mathcal{Z}(\Vect_G^\ThreeCocycle)$}

Let $G$ be a finite group and $\ThreeCocycle \in Z^3(G,\Tt)$ a 3-cocycle. 
\begin{defi}
A monomial Yetter-Drinfeld module  is a monomial space $\bV=(V,X,(V_x)_{x\in X})$ such that $V\in \mathcal{Z}(\Vect_G^\ThreeCocycle)$, the twisted $G$-action $\Acts$ permutes the $V_x$'s and each $V_x$ is $G$-homogeneous.
\end{defi}

\begin{rem}
\begin{enumerate}[leftmargin=*,label=\rm{(\alph*)}]
    \item If $\bV=(V,X,(V_x)_{x\in X})$ is a monomial Yetter-Drinfeld module, the set $X$ is a crossed $G$-set with the induced $G$-action and the grading map.
    \item If $\bV=(V,X,(V_x)_{x\in X})$ is a monomial Yetter-Drinfeld module, the action of $G$ on  $(V_e,X_e,(V_x)_{x\in X_e})$ is monomial, where $X_e:=\{x\in X: |x|=e\}$ and $V_e=\oplus_{x\in X_e}V_x$.
\end{enumerate}
\end{rem}

\begin{thm}\label{thm: YD monomial }
Let $G$ be a finite group, $\omega \in Z^3(G,\Tt)$. If $\bV=(V,X,(V_x)_{x\in X})$ is a monomial Yetter-Drinfeld module in $\mathcal{Z}(\Vect_G^\ThreeCocycle)$, then
\begin{enumerate}[leftmargin=*,label=\rm{(\alph*)}]
    \item  the action of $\mathcal{B}_n$ on $\Hom_{\mathcal{Z}(\Vect_G^\ThreeCocycle)}(\mathbb{C}, V^{\otimes n} )$ is monomial,
    \item the dimension of $\Hom_{\mathcal{Z}(\Vect_G^\ThreeCocycle)}(\mathbb{C}, V^{\otimes n} )$ is equal to the number of regular $G$-orbits under the monomial action of $G$ on \[(X^{ n})_e:=\{(x_1,\ldots,x_n): |x_1|\cdots |x_n|=e\}.\]
\end{enumerate}
\end{thm}
\begin{proof}
The action of $G$ on  $(V^{\otimes n}_e,(X^{n})_e,(V_x)_{x\in X_e})$ is monomial. Hence by Proposition \ref{Prop:Lemma Karp}, the triple 
\[
\bV_e^G := \left( ((V^{\otimes n}_e)^G, \widetilde{(X^{n})_e}, \big (\Av_G((V^{\otimes n}_e)_\Or)\big)_{\Or \in \widetilde{(X^{n})_e}} \right)
\]
is a monomial space. Since $\Hom_{\mathcal{Z}(\Vect_G^\ThreeCocycle)}(\mathbb{C}, V^{\otimes n} )= (V^{\otimes n})_e^G$, and each of the automorphisms  $\sigma'$ are morphisms in  $\mathcal{Z}(\Vect_G^\ThreeCocycle)$, hence $\sigma'|_{V^{\otimes n}_e}:(V^{\otimes n}_e,(X^{n})_e,(V_x)_{x\in X_e})\to (V^{\otimes n}_e,(X^{n})_e,(V_x)_{x\in X_e})$ is a $G$-linear isomorphism of monomial spaces. It follows from Proposition \ref{iso of invariants} that $\sigma'|_{(V^{\otimes n})_e^G}$ is an isomorphism of monomial spaces. Thus, the linear representation \begin{align*}
    \rho_n: \mathcal{B}_n&\to \operatorname{GL}((V^{\otimes n}_e)^G)\\
    \sigma &\mapsto \sigma',
\end{align*} is a monomial representation of $\mathcal{B}_n$.
The second part follows immediately from Proposition \ref{Prop:Lemma Karp}.
\end{proof}

\subsection{Monomial matrices of the braid representation}\label{subsection:monomial basis}

In this subsection we  obtain concrete formulas for the monomial braid representations associated to a monomial Yetter-Drinfeld module.

Let $G$ be a finite group, $\omega \in Z^3(G,\Tt)$ and $\bV=(V,X,(V_x)_{x\in X})$ be a monomial Yetter-Drinfeld module. If we fix non-zero vectors $\mathcal{S}:=\{v_x\in V_x:x\in X\}$, the twisted $G$-action defines a map \[\lambda_X:G\times X\to \Tt,\]
by $g \Acts v_x= \lambda_X(g;x)v_{gx}$, where $g \in G, x\in X$. 

For the monomial Yetter-Drinfeld module $\bV^{\otimes n}=(V^{\otimes n} ,X^{n},(V_x)_{x\in X^{n}})$ and the basis  $\mathcal{S}^{\otimes n}:=\{v_{x_1}\otimes \cdots \otimes v_{x_n}:x_i\in X, \ 1\leq i\leq n\}$, the 
action is determined by the map $\lambda_{X^n}:G\times X^n\to \Tt,$
\begin{align}
    \lambda_{X^n}(g;x_1,\ldots,x_n)&:=\prod_{i=1}^n\lambda_X(g;x_i)\ThreeCocycle(g;|x_1||x_2|\cdots |x_{n-1}|,|x_n|)\times \label{definition of lambda x n}\\ &\ThreeCocycle(g;|x_1|\cdots  |x_{n-2}|,|x_{n-1}|)
    \cdots \ThreeCocycle(g;|x_1|,|x_2|), \notag
\end{align}that is, \[g\rhd  (v_{x_1}\otimes \cdots \ot v_{x_n})= \lambda_{X^n}(g;x_1,\ldots,x_n) (v_{gx_1}\otimes \cdots \ot v_{gx_n}),\] for all $g\in G, x_1,x_2\ldots , x_n \in X$. Hence an element $(x_1,\ldots, x_n) \in (X^{n})_e$ is regular if and only if \begin{align}\label{condicion regular}
\lambda_{X^n}(g;x_1\ldots, x_n)=1,&& \text{ for all } g \in \bigcap_{i=1}^n \stab(x_i).    
\end{align}

Let $\R\subset X^n_e$ be a set of representatives of the regular orbits of $X_e^{\times n}$. Let $\mathcal{S}_{reg}= \{v_{x_1}\ot \cdots \ot v_{x_n}: (x_1,\ldots,x_n)\in R \}$.   By Proposition \eqref{Prop:Lemma Karp} the set $\{\Av_G(v): v\in \mathcal{S}_{reg}\}$ is a basis of $(V^{\otimes n})_e^G$.

In order to express the  action of the generator $\sigma_i \in \mathcal{B}_n$ in terms of $\{\Av_G(v): v\in \mathcal{S}_{reg}\}$, for each $\mathbf{x}=(x_1,\ldots,x_n)\in \R$ choose  $g_{\mathbf{x}}\in G$ such that $g_{\mathbf{x}}\Acts \sigma_i'(\mathbf{x})=\mathbf{y}$, where $\mathbf{y}\in \R$ and $\sigma_i'(\mathbf{x})= (x_1, \cdots , x_{i-1}, |x_i|x_{i+1}, x_i, \cdots ,  x_n)$. Hence there is  $\beta_{i,\mathbf{x}}\in \Tt$ such that $g_{\mathbf{x}}\Acts\sigma_i'(v_{x_1}\ot \cdots v_{x_n})=\beta_{i,\mathbf{x}}v_{y_1}\ot \cdots \ot v_{y_n}$.

Since the action of the generator $\sigma_i \in \mathcal{B}_n$ is given by \begin{align}
\sigma'(v_{x_1}\ot \cdots v_{x_n})&= \ThreeCocycle^{-1}(|x_1|\cdots |x_{i-1}|,|x_{i+1}|,|x_i|)\times \label{formula generador grupo de trenza}\\ 
&\lambda_X(|x_i|;x_{i+1})\ThreeCocycle(|x_1|\cdots |x_{i-1}|,|X_i|,|x_{i+1}|)\times \notag\\ & v_{x_1}\ot \cdots \ot v_{x_{i-1}}\ot v_{|x_i|x_{i+1}}\ot v_{x_i}\ot \cdots \ot  v_{x_n}.\notag
\end{align}
we have that 
\begin{align}
  \beta_{i,\mathbf{x}}&=  \ThreeCocycle^{-1}(|x_1|\cdots |x_{i-1}|,|x_{i+1}|,|x_i|)\times\label{def:beta} \\ 
&\lambda_X(|x_i|;x_{i+1})\ThreeCocycle(|x_1|\cdots |x_{i-1}|,|X_i|,|x_{i+1}|)\lambda_{X^n}(g_{\mathbf{x}};\sigma_i'(\mathbf{x})).\notag
\end{align}

\begin{thm}\label{thm:action in basis}
Let $G$ be a finite group, $\omega \in Z^3(G,\Tt)$ and $\bV=(V,X,(V_x)_{x\in X})$ be a monomial Yetter-Drinfeld module. Let $Y$ be the set of all regular elements in $X^n_e$ and $\R\subset Y$ a set of representatives of the $G$-orbits of $Y$.

\begin{enumerate}[leftmargin=*,label=\rm{(\alph*)}]
\item The projection $\pi :Y\to \R$ is map of $\mathcal{B}_n$-sets. The image of $\mathbf{x}\in \R$ by the generator $\sigma_i\in \mathcal{B}_n$ will be denoted by $\sigma_i\Acts \mathbf{x}$.
\item  Let $\mathcal{S}_{reg}= \{v_{x_1}\ot \cdots \ot v_{x_n}: (x_1,\ldots,x_n)\in \R \}$.  The action of the generator $\sigma_i \in \mathcal{B}_n$ in the basis $\{\Av_G(v_{\mathbf{x}}): \mathbf{x}\in \R\}$ is given by \[\sigma_i(\Av_G(v_{\mathbf{x}}))= \beta_{i,\mathbf{x}}\Av_G(v_{\sigma_i \Acts \mathbf{x}}),\]where $\beta_{i,\mathbf{x}}$ was defined in \eqref{def:beta}.
\end{enumerate}
\end{thm}
\begin{proof}
The first part is consequence of Theorem \ref{thm: YD monomial }. 

For the second part, recall that the number $\beta_{i,\mathbf{x}}$ and the element $g_\mathbf{x}\in G$ are such that \[g_\mathbf{x}\Acts \sigma(v_{\mathbf{x}})=\beta_{i,\mathbf{x}}v_{\sigma_i\Acts \mathbf{x}}.\] Hence,
\begin{align*}
    \sigma_i(\Av_G(v_{\mathbf{x}}))&= \Av_G(\sigma_i(v_{\mathbf{x}}))\\
    &=g_\mathbf{x}\Acts \Av_G(\sigma_i(v_{\mathbf{x}}))\\
    &=\Av_G(g_\mathbf{x}\Acts\sigma_i(v_{\mathbf{x}}))\\
    &=\Av_G(\beta_{i,\mathbf{x}}v_{\sigma_i\Acts \mathbf{x}})\\
    &=\beta_{i,\mathbf{x}}\Av_G(v_{\sigma_i \Acts \mathbf{x}}).
\end{align*}
\end{proof}
\begin{example}\label{ejemplo permutation crossed g-sets}
Let $G$ be a finite group and $X$ be a left crossed $G$-set. Then the linearization $V_X:=\oplus_{x\in X}\mathbb{C}x$ is a (untwisted) Yetter-Drinfeld module in  $\mathcal{Z}(\Vect_G)$. Clearly $\lambda_X\equiv 1$, thus every element in $(X^{n})_e$ is regular. Hence the canonical projection 
\[(X^{ n})_e\to (X^{ n})_e//G,\]is an epimorphism of $\mathcal{B}_n$-sets. In other words, the linear representation of $\mathcal{B}_n$ on $\Hom_{\mathcal{Z}(\Vect_G)}(\mathbb{C}, V_X^{\otimes n} )$ is the linearization of the permutation action of $\mathcal{B}_n$ on $(X^{ n})_e//G$.
\end{example}
\section{Braid groups representations Associated to Lagrangian algebras}\label{Section: representation Lagrangian}

In this section we prove that every Lagrangian algebra $\mathcal{Z}(\Vect_G^\ThreeCocycle)$ as a canonical monomial structure. Then the results of Section \ref{Section: monomial twisted Yetter-Drinfeld} can be applied to Lagrangian algebras in $\mathcal{Z}(\Vect_G^\ThreeCocycle)$.

\subsection{Lagrangian algebras }

Following \cite[Corollary 3.17]{davydov2017lagrangian}, we will describe the Lagrangian algebra on $\twistedCenterG$ associated to a pair $(H,\gamma)$, where $H\subseteq G$, is a subgroup and $\gamma:H\times H\to \Tt$  a map such \begin{align*}
    \frac{\gamma(ab,c)\gamma(a,b)}{\gamma(a,bc)\gamma(b,c)}=\ThreeCocycle(a,b,c),&& a,b,c \in H.
\end{align*} 
Let  $\tCH=\oplus_{h\in H}\mathbb{C}e_h$   the group algebra of $H$ with the multiplication \begin{align*}
    e_{h_1}e_{h_2}=\gamma(h_1,h_2)e_{h_1h_2},&& h_1,h_2 \in H.
\end{align*} The vector space $\tCH=\oplus_{h\in H}\mathbb{C}e_h$, is a commutative algebra in $\mathcal{Z}(\Vect_H^\ThreeCocycle)$, where the  $H$-action 
is given by
\begin{align*}
    h_1\Acts e_{h_2}=\eps(h_1,h_2)e_{h_1h_2h_1^{-1}},&& \eps(h_1,h_2):=\frac{\gamma(h_1,h_2)}{\gamma(\prescript{h_1}{}{h_2},h_1)},&& h_1,h_2 \in H,
\end{align*}and grading $|e_h|=h$ for all $h\in H$.

Let $\operatorname{Map}(G, \tCH)$ be
the vector space of all set-theoretic maps from $G$ to  $\tCH$. With the grading given by 

\begin{align*}
\vert a\vert&=f &\Leftrightarrow &&\forall x\in G \quad \vert a(x)\vert =x^{-1}fx,   
\end{align*}
and twisted $G$-action  

\begin{align*}\label{G-action}
(g\Acts a)(x):=\ThreeCocycle(x^{-1},g^{-1}; |a|)^{-1}a(g^{-1}\Acts x), && g,x\in G.
\end{align*}
$\operatorname{Map}(G, \tCH)$ is twisted Yetter-Drinfeld module. 

The Lagrangian algebra $L(H,\gamma)$ is the Yetter-Drinfeld submodule 
\begin{align*}
    L(H,\gamma):=\{a\in \textrm{Maps}(G,\tCH)\vert \, a(xh)=\ThreeCocycle(h^{-1},x^{-1}; |a|)h^{-1}\Acts a(g) \},
\end{align*}see \cite{davydov2017lagrangian} for more details.

\subsection{Monomial structure of the Lagrangian algebras $L(H,\gamma)$}
In this section we will proved that every Lagrangian algebra of the form $L(H,\gamma)$ has a canonical monomial structure.

Let $G$ be a group and $H\subset G$ a subgroup. We can regard $G\times H$ as a left $H$-set  with actions given $h\Acts (g,h')=(gh^{-1},hh'h^{-1})$. Then we can consider the set of $H$-orbits  that we will denote by $G \times_H H$. The set  $G \times_H H$ is equipped with a left $G$-action given by left multiplication on the first component.

\begin{defi}
Let $L(H,\gamma)$ be a Lagrangian. For each $g\in G$, $f\in H$, define $\chi_{g,f}\in L(H,\gamma)$ by 

\begin{equation}
\chi_{g,f}(x)=
\begin{cases}
0, & x\notin gH \\
\ThreeCocycle(h^{-1},g^{-1}; \prescript{g}{}{f})\epsilon(h^{-1},f)e_{hfh^{-1}}, & x=gh, \text{ where } h\in H.
\end{cases}
\end{equation}
\end{defi}
\begin{rem}
The function $\chi_{g,h}$ can be  characterized as the unique map in $L(H,\gamma)$ with support $gH$ and such that $\chi_{g,h}(g)=e_h$.
\end{rem}
\begin{lem}\label{Lemma: equaciones}
Let $L(H,\gamma)$ be a Lagrangian algebra in $\twistedCenterG$. Then 
 \begin{align}\label{equ1: chi}
\chi_{gh,f}=\ThreeCocycle(h,(gh)^{-1}; \prescript{gh}{}{f})\epsilon(h,f)\chi_{g,\prescript{h}{}{f}},&& g\in G, f,h\in H.
\end{align}
\begin{align}\label{equ2: chi}
    l\Acts \chi_{g,f}=\ThreeCocycle((lg)^{-1},l^{-1};\prescript{g}{}{f})\chi_{lg,f},&& g,l \in G, h\in H.
\end{align}
\end{lem}
\begin{proof}
\eqref{equ1: chi}. Since the supports of  $\chi_{gh,f}$ and $\chi_{g,^hf}$ are $gH$, and

\begin{align*}
\chi_{gh,f}(g) &= \chi_{gh,f}(ghh^{-1})\\
&= \ThreeCocycle(h,(gh)^{-1}; \prescript{gh}{}{f})\epsilon(h,f)\chi_{g,\prescript{h}{}{f}}(g),
\end{align*} we obtain \eqref{equ1: chi}.

\eqref{equ2: chi}. By the definition of the action of $G$ we have \begin{align*}
    l\Acts \chi_{g,f}(lg)&=\ThreeCocycle((lg)^{-1},l^{-1};\prescript{g}{}{f})\chi_{g,f}(g)\\
    &=\ThreeCocycle((lg)^{-1},l^{-1};\prescript{g}{}{f})e_f.
\end{align*}Since $l\rhd \chi_{g,f}$ and $\chi_{gl,f}$ are supported in $glH$, we get \eqref{equ2: chi}.
\end{proof}
It follows from Lemma \ref{Lemma: equaciones} that $\mathbb{C}\chi_{gh,\prescript{h}{}{f}}=\mathbb{C}\chi_{g,f}$. Then for any $(g,h)\in G\times_H H$ the space $\mathbb{C}\chi_{g,f}$ is well defined.
\begin{thm}\label{thm: Lagrangian are monomial}
Let $L(H,\gamma)$ be a Lagrangian algebra  in $\twistedCenterG$. Then $L(H,\gamma)$ with the decomposition \[L(H,\gamma)=\bigoplus_{(g,h)\in G\times_H H}\mathbb{C}\chi_{g,h}\] is a monomial twisted Yetter-Drinfeld module. 
\end{thm}
\begin{proof}
First we will check that in fact the sum $\sum_{(g,h)\in G\times_H H}\mathbb{C}\chi_{g,h},$ is direct. Since $\textrm{supp}(\chi_{g,f})=gH$, we have that  $\chi_{g,f}$ and $\chi_{g^\prime,f}$ are linearly independent if $gH\neq g^\prime H$. Hence it is suffices to check linear independence of the collections $\{\chi_{g,f}\}_{f\in H}$, with $g$ fixed. But if $f\neq f^\prime$, $\vert \chi_{r,f}\vert \neq \vert \chi_{r,f^\prime}\vert$. It follows that the sum $\sum_{(g,h)\in G\times_H H}\mathbb{C}\chi_{g,h}$ is direct.

In order to see that $L(H,\gamma)=\sum_{(g,h)\in G\times_H H}\mathbb{C}\chi_{g,h}$, fix $\R\subset G$ a set of representative of the left coset of $H$ in $G$. Let $a\in L(H,\gamma)$. For each $r\in \R$, suppose 
\begin{equation}
a(r)=\sum_{f\in H} \lambda_{r,f} e_f. 
\end{equation}
Then we have
\begin{equation}
a=\sum_{r\in \R,f\in H} \lambda_{r,f}\chi_{r,f}\in \sum_{(g,h)\in G\times_H H}\mathbb{C}\chi_{g,h}.
\end{equation}

By \eqref{equ2: chi}  and the fact that $|\chi_{g,f}|=gfg^{-1}$, we obtain that $L(H,\gamma)$ is a monomial twisted Yetter-Drinfeld module.
\end{proof}

\begin{cor}\label{cor: Lagrangian are monomial }
Let $G$ be a finite group, $\omega \in Z^3(G,\Tt)$. If $L(H,\gamma)$ is a Lagrangian algebra in $\mathcal{Z}(\Vect_G^\ThreeCocycle)$, then
\begin{enumerate}[leftmargin=*,label=\rm{(\alph*)}]
    \item  the action of $\mathcal{B}_n$ on $\Hom_{\mathcal{Z}(\Vect_G^\ThreeCocycle)}(\mathbb{C}, L(H,\gamma)^{\otimes n} )$ is monomial,
    \item the dimension of $\Hom_{\mathcal{Z}(\Vect_G^\ThreeCocycle)}(\mathbb{C}, L(H,\gamma)^{\otimes n} )$ is equal to the number of regular $G$-orbits under the monomial action of $G$ on \[(G\times_H H)^{\times n})_e:=\{((g_1,h_1),\ldots,(g_n,h_n)): g_1h_1g_1^{-1}g_2h_2g_2^{-1}\cdots g_nh_ng_n^{-1}=e\}.\]
\end{enumerate}
\end{cor}
\begin{proof}
Follows from Theorem \ref{thm: Lagrangian are monomial} and Theorem \ref{thm: YD monomial }.
\end{proof}

We will fix a set of representatives of the left cosets of $G$ in $H$, $\R \subset G$. Thus
every element $g \in G$ has a unique factorization $g = rh$, $h \in H, r \in \R$. We assume $e \in \R$. The uniqueness of the factorization $G = \R H$ implies that there are well defined maps
\begin{align*}
\Acts: G\times \R\ra \R,&& \kappa:G\times \R\ra H, 
\end{align*}
determined by the condition
\begin{align*}
    gr=(g\Acts r)\kappa(g,h), && g\in G, r\in \R.
\end{align*}

As crossed $G$-set we can identify $G\times_H H$ with  $\R\times H$ with action 
\begin{align*}
    g\Acts(r,h):= (g\Acts r,^{\kappa(g,r)}h),&&  r\in \R,\  h\in H,\  g\in G,
\end{align*}
and grading map 
\begin{align*}
|-|:\R\times H&\to G\\
(r,h)&\mapsto rhr^{-1}.
\end{align*}

It follows from Theorem \eqref{thm: Lagrangian are monomial} that $B_\R:=\{\chi_{r,h} | \, r\in \R, \, h\in H\}$ is a basis for $L(H,\gamma)$. 

In order to apply the results of Subsection \ref{subsection:monomial basis}, we only need to compute the map $\lambda_{\R\times H}:G\times (\R\times H)\to \Tt$, such that  

\begin{align*}
    g \Acts \chi_{r,h}= \lambda_{\R\times H}(g;r,h) \chi_{g\Acts r,\prescript{\kappa(g,r)}{}{h}}, && g\in G, r\in \R, h\in H.
\end{align*}

Using Lemma \ref{Lemma: equaciones}
we obtain that 
\begin{equation}\label{LambdaCoef}
\lambda_{\R\times H}(g;r,h)=
 \ThreeCocycle((g r)^{-1},g^{-1}; \prescript{r}{}{h})
  \ThreeCocycle(\kappa(g,r),(gr)^{-1}; \prescript{gr}{}{h})
   \eps(\kappa(g,r),h),
\end{equation}for all $g\in G, r\in \R, h\in H$.

By \eqref{condicion regular},we have that an element $t=((r_1,h_1),\ldots,(r_n,h_n))\in (\R\times H)^n_e$ is regular if and only if 
\begin{align}\label{regular lagrangian}
 \lambda_{(\R\times H)^n}(g;(r_1,h_1),\ldots,(r_n,h_n))=1,&&  \text{for all } g\in \bigcap_{i=1}^{n} r_i^{-1}C_H(h_i)r_i
\end{align}
where $\lambda_{(\R\times H)^n}$ was defined in \eqref{definition of lambda x n} in function of $\lambda_{\R\times H}$ and $\ThreeCocycle$.

\subsection{Applications and examples}
In this last section we present some application of the results of the previous section.

\subsubsection{Central Subgroups}
\begin{prop}
Let $G$ be a finite group and  $L(H,\gamma)$ a Lagrangian algebra in $\mathcal{Z}(\Vect_G)$, where  $H\subset G$ is a central subgroup.  Then
 \[\dim\Big (\Hom_{\mathcal{Z}(\Vect_G)}(\mathbb{C},L(H,\gamma)^{\otimes n})\Big ) =|G|^{n-1}.\]  Moreover, the representation of $\B_n$ is actually a representation of $\SymGp_n$.
\end{prop}
\begin{proof}
Since $H$ is a central subgroup, $g\Acts (r,h) = (g\Acts r, h)$ and 
\begin{equation*}
    |\chi_{r_1,h_1} \otimes \dots \otimes \chi_{r_k,h_k}| = h_1 \cdots h_k,
\end{equation*}
for any  $ r_1,\dots,r_k \in R, h_1, \dots, h_n \in H$.  
Hence,
\begin{align*}
    |(\R\times H)^n_e|&=|(\R^n/G)| |H^{n-1}|\\
    &= [G:H]^n|H|^{n-1}\\
    &= |G|^{n-1}.
    \end{align*} 
 
To determine the number of orbits, notice that $\eps: H\times H\to \Tt$ is a bicharacter such that $\eps(h_1,h_2)\eps(h_2,h_1)=1$. Then by equation \eqref{regular lagrangian} an element \[\Big( (r_1,h_1),\ldots , (r_n,h_n)\Big ) \in (\R\times H)_e^n\] is regular if and only if 
\begin{align*}
    \prod_{i=1}^n\eps(h,h_i)=1, && \text{for all } h\in H.
\end{align*}
But $\prod_{i=1}^n \eps(h,h_i) = \eps(h,h_1\cdots h_n)   = \eps(h', e) = 1.$ Hence every element is regular. By Corollary \ref{cor: Lagrangian are monomial } the dimension of $\Hom_{\mathcal{Z}(\Vect_G)}(\mathbb{C},L(H,\gamma)^{\otimes n})$  is  $|G|^{n-1}$.

Finally, using equation \eqref{formula generador grupo de trenza}, we  see that 
\begin{align*}
\sigma'_i\circ\sigma'_i (\chi_{r_1,h_1}\otimes\cdots \ot \chi_{r_n,h_n})&= \eps(h_i,h_{i+1})\eps(h_{i+1},h_{i}) (\chi_{r_1,h_1}\otimes\cdots \ot \chi_{r_n,h_n})\\ 
&= \chi_{r_1,h_1}\otimes\cdots \ot \chi_{r_n,h_n}.
\end{align*}
Hence  representation of $\mathcal{B}_n$ factors as a representation of $\SymGp_n$.

\end{proof}

\subsubsection{Lagrangian algebra of the form $L(H,1)$} 
The Lagrangian algebras $L(H,1)$ as an object in $\mathcal{Z}(\Vect_G)$ are completely determined by the crossed $G$-set $G\times_H$, and the monomial representation $\Hom(\mathbb{C}, L(H,1)^{\otimes})$ is a permutation representation, see Example \ref{ejemplo permutation crossed g-sets}.  Let us see some extreme cases:

\subsubsection*{Case $H=\{e\}$}
In this case the crossed  $G$-set is $G$ with the regular action and grading map the constant map  $e$. It is clear that the braiding $c_{G,G}$ 
is just the flip map \[(g_1,g_2)\mapsto (g_2,g_1)\]hence, really the symmetric group $\SymGp_n$ acts on $G^n$.

The set of  $G$-orbits is in biyection with $G^{n-1}$,  \begin{align*}
    \mathcal{O}(G^n)&\to G^{n-1}\\
    \mathcal{O}_G(g_1,g_2,\ldots, g_n)&\mapsto (e,g_1^{-1}g_2,\ldots, g_1^{-1}g_n).
\end{align*} 
Using the previous map the action of $\S_n$ is given by \[\sigma_1\Big( g_1,\ldots, g_{n-1}\Big)= \Big ( g_1^{-1},g^{-1}_1g_2,\ldots, g_1^{-1}g_{n-1}\Big)\]

and \[\sigma_i(g_1,\ldots, g_i,g_{i+1},\ldots,g_{n-1})=(g_1,\ldots,g_{i+1},g_i,\ldots,g_{n-1}), \quad 1<i<n.\]
It is clear that permutation action of $\SymGp_n$ on $G^{n-1}$ is faithful, thus the image is isomorphic to $\SymGp_n$.

\subsubsection*{Case $H=G$} In this case the crossed $G$-set is $G$ with the action by conjugation and grading map the identity map. Hence, the braiding is given by
\[c_{G,G}:(x,y)\mapsto (y,y^{-1}xy).\]
Note $c_{G,G}$  is symmetric if and only if  $G$ is abelian. 

If $G$ is abelian,  $G^n_e=\{(g_1,\ldots,g_{n-1},(g_1\ldots g_{n-1})^{-1})\}$ is the set of orbits and as the previous example the group $\SymGp_n$ acts faithfully.

\subsubsection{Dihedral group}

Every time we take $H$ to be a normal subgroup of $G$, the following proposition provides a way to simplify the situation. 

\begin{prop}\label{decoupledBasis}
Let $G$ be a finite group, $H\Normal G$, $\R$ a collection of representatives for $G/H$. Define $B_\gamma[H]\in\mathcal{Z}(\Vect_G)$ as 
\begin{equation*}
    B(H,\gamma) := \Span \{ b_{r,h} |\, r\in \R, h\in H \} 
\end{equation*}
with grading $|b_{r,h}|=h$ and the $G$-action
\begin{equation}\label{actionOnTheBs}
   g\Acts b_{r,h} = \eps(\kappa(g,r)\prescript{r^{-1}}{}{h}) b_{g\Acts r, \prescript{g}{}{h}}.
\end{equation}
Then, the mapping 
\begin{align*}
     B(H,\gamma) &\ra L(H,\gamma) \\
    b_{r,h} &\mapsto \chi_{r, \prescript{r^{-1}}{}{h}}
\end{align*}
is an isomorphism in $\mathcal{Z}(\Vect_G)$. 
\end{prop}

\begin{proof}
We need to show the map preserves the grading and the $G$-representation. We have 
\begin{equation*}
    |\chi_{r,\prescript{r^{-1}}{}{h}}|=\prescript{r}{}{(\prescript{r^{-1}}{}{h})} = h = |b_{r,h}|
\end{equation*}
Now, since
\begin{equation*}
    g\cdot \chi_{r,\prescript{r^{-1}}{}{}h} = \eps(\kappa(g,r),\prescript{r^{-1}}{}{h})\chi_{g\Acts r, \prescript{\kappa(g,r)}{}{(\prescript{r^{-1}}{}{h})} }  ,               
\end{equation*}
and
\begin{equation*}
    \prescript{\kappa(g,r)}{}{(\prescript{r^{-1}}{}{h})} = \prescript{(g\Acts r)^{-1}}{}{ghg^{-1}},
\end{equation*}
we have that 
\begin{equation*}
    g\Acts b_{r,h} = \eps(\kappa(g,r),\prescript{r^-1}{}{h}) a_{g\Acts r, \prescript{(g\Acts r)^{-1}}{}{(\prescript{g}{}{h})}}.
\end{equation*}
Hence by \eqref{actionOnTheBs}  the map is equivariant. 
\end{proof}

Proposition \ref{decoupledBasis} works particularly well when $\gamma =1$, since equation \eqref{actionOnTheBs} is just 
\begin{equation*}
    g\Acts b_{r,h} = b_{g\Acts r, \prescript{g}{}{h}}.
\end{equation*}
Thus, the action of $G$ is "decoupled".  We use this idea in the following example. 

Let $G=D_{2k}$ be the dihedral groups of order $2k$  and $H=\langle r \rangle$. We take $\R=\{e,s\} = \{s^i\}_{i\in \ZZ/2\ZZ}$. 
Then 
\begin{equation*}
    |b_{s^{i_1},r^{j_1}}\otimes \cdots \otimes b_{s^{i_n},r^{j_n}}| = r^{\sum_{m=1}^n j_m},
\end{equation*}
and
\begin{equation*}
    \dim( B(H,\gamma)^{\otimes n}_e) = 2^n \times k^{n-1}.
\end{equation*} 
Since 
\begin{equation*}
    (s^i r^j )(s^k) = s^{i+k} r^{(-1)^k j}, 
\end{equation*}
we have that
\begin{align*}
    (s^i r^j) \Acts s^k = s^{i+k},&&  \text{ and }& 
    &\kappa(s^ir^j,s^k) = r^{(-1)^kj}.
\end{align*}
Hence, the action, on the set label is 
\begin{equation*}
    s^i r^j (s^k,r^l) = (s^{i+k},r^l)
\end{equation*}
It follows that the number of orbits in $(\R\times H)^n_e$ is 
\begin{equation*}
    2^{n-1}\times k^{n-1} = |G|^{n-1}.
\end{equation*}
Since $\gamma = 1$ all orbits are regular and then  $\dim(\Hom_{\mathcal{Z}(\Vect_G)}(\mathbb{C},L(H,1)^{\otimes n}))=|G|^{n-1}$.

\end{document}